\documentclass[11pt, a4paper]{amsart}

\usepackage{a4wide}
\usepackage[all,cmtip]{xy}

\usepackage{numprint}

\usepackage[english]{babel}
\usepackage{scalerel, mathtools}
\usepackage[dvipsnames]{xcolor}
\usepackage{hyperref}
\usepackage{microtype}
\usepackage{fancyhdr}
\usepackage{amsmath, amssymb,color}
\usepackage{amsthm}

\usepackage[T1]{fontenc}
\usepackage{lmodern}\normalfont
\DeclareFontShape{T1}{lmr}{bx}{sc} { <-> ssub * cmr/bx/sc }{}

\newcommand{\Lg}{\mathfrak{g}}

\newcommand{\Lie}{\mathfrak{Lie}}

\newcommand{\im}{\mathop\mathrm{im}}

\newcommand{\spn}{\mathop\mathrm{span}}

\newcommand{\ov}{\overline}

\newcommand*\DA{\mathop{}\!\mathbin\Box}

\newcommand{\cp}{\ov{\partial}}

\newcommand{\e}{\mathrm{e}}

\newcommand{\del}{\partial}

\newcommand{\Aut}{\mathrm{Aut}}

\newcommand{\aut}{\mathfrak{aut}}
\newcommand{\id}{\operatorname{\text{\sf id}}}
\renewcommand{\Re}{\operatorname{Re}}

\newcommand{\Oo}{\mathcal{O}}
\newcommand{\LL}{\mathcal{L}}
\newcommand{\TT}{\mathbb{T}}

\newcommand{\al}{\alpha}
\newcommand{\be}{\beta}

\newcommand{\la}{\lambda}

\renewcommand{\phi}{\varphi}

\newcommand{\ce}{\mathcal{C}^\infty}
\newcommand{\CC}{\mathbb{C}}

\newcommand{\RR}{\mathbb{R}}
\newcommand{\Ss}{\mathbb{S}}
\newcommand{\KK}{\mathbb{K}}
\newcommand{\ZZ}{\mathbb{Z}}

\newcommand{\QQ}{\mathbb{Q}}

\newcommand{\Ka}{K\"{a}hler}

\newcounter{Mycounter}[section]
\newcounter{lemma}[section]
\newcounter{claim}[section]

      \newtheorem{clm}{Claim}

\newcounter{fact}[section]

\newcounter{sublemma}[section]

\newcounter{corollary}[section]

\newcounter{theorem}[section]

\newcounter{conjecture}[section]

\newcounter{proposition}[section]

\newcounter{definition}[section]

\newcounter{example}[section]

\newcounter{remark}[section]

\newcounter{problem}[section]

\newcounter{question}[section]

 \stepcounter{thmx}

\makeatletter

\@addtoreset{equation}{section}

\@addtoreset{footnote}{section}

\makeatother

\usepackage{enumerate}
\def\eqref#1{(\ref{#1})}

\def\1{\sqrt{-1}\:}

\newcommand{\cntrct}                % contraction with a vector field
{\hspace{2pt}\raisebox{1pt}{\text{$\lrcorner$}}\hspace{2pt}}

\makeatletter
\newcommand*\rel@kern[1]{\kern#1\dimexpr\macc@kerna}
\newcommand*\widebar[1]{%
  \begingroup
  \def\mathaccent##1##2{%
    \rel@kern{0.8}%
    \overline{\rel@kern{-0.8}\macc@nucleus\rel@kern{0.2}}%
    \rel@kern{-0.2}%
  }%
  \macc@depth\@ne
  \let\math@bgroup\@empty \let\math@egroup\macc@set@skewchar
  \mathsurround\z@ \frozen@everymath{\mathgroup\macc@group\relax}%
  \macc@set@skewchar\relax
  \let\mathaccentV\macc@nested@a
  \macc@nested@a\relax111{#1}%
  \endgroup
}
\makeatother
\makeatletter
\def\cleardoublepage{\clearpage\if@twoside \ifodd\c@page \else\hbox{}\thispagestyle{empty}\newpage
\if@twocolumn\hbox{}\newpage\fi\fi\fi} \makeatother
\title{LCK metrics on toric LCS manifolds}
\author{Nicolina Istrati}
\address{School of Mathematical Sciences, Tel Aviv University, Ramat Aviv, Tel Aviv 69978 Israel}

\email{nicolinai@mail.tau.ac.il}

\begin{document}
\begin{abstract}
We show a bijective correspondence between compact toric locally conformally symplectic manifolds which admit a compatible complex structure and  pairs $(C,a)$, where $C$ is a \textit{good cone} in the dual Lie algebra of the torus and $a$ is a positive real number. Moreover, we prove that any toric locally conformally K\"ahler metric on a compact manifold admits a positive potential.
\end{abstract}

\maketitle

%\tableofcontents

\section{Introduction}

A locally conformally symplectic (LCS) form on a manifold $M$ is a non-degenerate two-form which, around any point of the manifold, is conformal to a locally defined symplectic form. This is equivalent to the fact that on a (usually infinite) covering of $M$, there exists a global symplectic form on which the deck group acts by strict homotheties. If additionally there exists a compatible integrable complex structure,  then the LCS form is called locally conformally \Ka\ (LCK). As such, one can think of LCS and LCK geometry as a twisted (or conformal) version of the more common symplectic and \Ka\ geometries.

In particular, it makes sense to talk about (twisted) Hamiltonian actions in this context and one can thus define toric LCS or LCK manifolds, in analogy with toric symplectic manifolds. These type of actions were introduced and motivated by Vaisman in \cite{v85}, and then considered again by Haller and Rybicki in \cite{hr}, by Gini, Ornea and Parton in \cite{gop} and by Stanciu in \cite{s} in the context of the reduction procedure for LCS/LCK manifolds. 

However, a more systematic study towards the classification of toric LCS manifolds has begun only recently. First, toric Vaisman manifolds were considered by Pilca in \cite{p16}, then compact toric LCK surfaces were classified, as complex manifolds, by Madani, Moroianu and Pilca in \cite{mmp}. In \cite{i17}, we showed that any compact toric LCK manifold admits a toric Vaisman metric. Finally, in \cite{bgp}, Belgun, Goertsches and Petrecca studied the moment map of a certain class of toric LCS manifolds and showed a corresponding convexity property.

This paper can be seen as a continuation of our previous paper on the same topic, and its goal is twofold. First, we give a classification of compact toric LCS manifolds admitting some compatible complex structure in terms of the cone over their moment map and a positive number (\ref{classifLCS}). The proof of this result has two ingredients: on one hand, we use the classification of a certain class of toric symplectic cones by the image of their moment maps. This image to which one adds a point forms a polyhedral cone with specific properties, called \textit{a good cone}. This description was achieved by Lerman \cite{l03} (see also Banyaga and Molino \cite{bm93}, \cite{bm96}, \cite{b99} and Boyer and Galicki \cite{bg00}, who had previously settled partial results in this direction). On the other hand, we derive, as a consequence of \cite{i17}, that the symplectic cover of any toric LCS manifold of LCK type admits the structure of a toric symplectic cone appearing in Lerman's classification (\ref{Vaismant}). 

One should note that not all toric LCS manifolds are of LCK type  (cf. \cite[Example~6.3]{i17}), and thus \ref{classifLCS} does not give a classification of all toric LCS manifold, as opposed to the classical Delzant classification of compact symplectic toric manifolds \cite{d}. However, all the examples of toric LCS manifolds that we are aware of arise from compact contact toric manifolds, classified also in \cite{l03}, so at this moment it is yet unclear how much bigger the class of all compact toric LCS manifolds is.

The second part is related to the \textit{LCK metrics with potential}. These are LCK metrics for which the corresponding \Ka\ metric on the cover admits a potential which is acted upon by homotheties by the deck group (see \ref{DefPot} for an equivalent definition). When the potential is strictly positive, the LCK metric is called with positive potential. This class of metrics was introduced by Ornea and Verbitsky in \cite{ov10} as a generalisation of the so called \textit{Vaisman metrics} (cf. \ref{DefVais}) and studied in subsequent papers \cite{ov12}, \cite{go14}, \cite{ovv} etc. In particular, in \cite{ov18} the authors show that a compact complex manifold with an LCK metric with potential admits an LCK metric with positive potential. Nonetheless, the question of whether any LCK metric with potential admits a positive potential has still remained open. In any case, it was observed by Vuletescu (see the introduction of \cite{ov18}) that an LCK form can have a non-positive potential, and more generally the potential need not be unique.

On the other hand, by combining \cite[Theorem~4.5]{llmp} and  \cite[Theorem~5.1]{t94}, it follows that any LCK metric on a compact Vaisman type manifold is exact, meaning that the corresponding \Ka\ metric admits a primitive on which the deck gorup acts by homotheties. It was subsequently believed that any LCK metric on such a manifold should also have a potential, but this was disproved by Goto \cite{go14}. However we still lack a good understanding of when this phenomenon occurs.

The second result of the present paper (\ref{toricPot}) states that any toric LCK metric on a compact manifold admits a positive potential. We deduce this from another result (\ref{APot}), giving sufficient conditions for an LCK metric on a compact Vaisman type manifold to admit a potential, and also for the uniqueness of specific potentials. The sufficient conditions consist in asking for a certain vector field - the anti-Lee vector field of any Vaisman metric (cf. \eqref{antiLee}) - to be an infinitesimal conformal symmetry of the metric. The questions of existence and uniqueness of the potential interpret in terms of the vanishing of certain cohomology classes. We show the desired vanishing by using Hodge theory with respect to a certain Laplacian. Finally, the positivity of the potential in the toric case easily follows as one is reduced to the simple study of convex functions on $\RR$ with an equivariance property. 

The paper is organised as follows: in Section \ref{SecDef} and \ref{SecToricLCS} we introduce the main definitions of LCS and LCK geometry, and then of toric LCS manifolds. In Section \ref{SecSymplCone} we discuss some definitions and properties of symplectic cones and then present Lerman's classification result. In Section \ref{SecClass} we prove our classification result (\ref{classifLCS}) of compact toric LCS manifolds of LCK type. Finally, in Section \ref{SecPot} we prove \ref{APot} and \ref{toricPot} regarding LCK metrics with potential.

\subsection*{Notation}
$G$ will always denote a compact $n$-dimensional torus, $\Lg$ its Lie algebra, $\Lg^*$ its dual Lie algebra and $\Lambda=\ker(\exp:\Lg\rightarrow G)$ its integral lattice. All group actions that we consider here are effective, so for $G$ acting on $M$ we will identify directly $\Lg$ with a subspace of $\ce(TM)$ and use $V$ to denote a vector field from $\Lg$. $C$ will be used to denote a cone in $\Lg^*$.  We will generally denote by $(M,\Omega)$ an LCS manifold, by $\theta$ the Lee form and  by $(\hat M,\omega)$ its minimal symplectic cover with deck group $\Gamma$. $A$ will always denote the anti-Lee vector field of $(\Omega,\theta)$, defined by $\iota_A\Omega=-\theta$, and $B=-JA$ the Lee vector field. $N$ will be used to denote a manifold which supports a symplectic cone structure, and $X$ will be used to denote a Liouville vector field. By $(S,\al)$ or by $(S,[\al])$ we will denote a contact manifold. $J$ will always denote an integrable complex structure.

The notation $\Omega^k(M,\KK)$ will be used for $\KK$-valued $k$-forms on $M$, where $\KK$ is either $\RR$ or $\CC$. When the field $\KK$ does not matter, we will sometimes also write $\Omega^k(M)$. Similarly, $\Omega^{p,q}(M,\CC)$ denotes the sheaf of $(p,q)$-forms on $(M,J)$. Finally, for a line bundle $L\rightarrow M$, $\Omega^\bullet(M)\otimes L$  will denote the sheaf of $L$-valued forms on $M$. For a vector field $U$, $\LL_U$ will denote the Lie derivative with respect to $U$ and $\iota_U$ the interior product with $U$.

\section{LCS and LCK structures}\label{SecDef}

We begin by recalling the definitions related to LCS and LCK geometry which are relevant in our context. For a more detailed account of the subject, the reader can consult \cite{do}.

Let $M$ be a connected compact manifold of real dimension $2n$, $n>0$. 

\begin{definition}
A real non-degenerate two form $\Omega$ on $M$ is called a \textit{locally conformally symplectic} (LCS) form if there exists a closed one-form $\theta$ on $M$, called the \textit{Lee form} of $\Omega$, such that:
\begin{equation}\label{defLCS}
d\Omega=\theta\wedge\Omega.
\end{equation}
If $\theta$ is not exact, then $\Omega$ is called a strict LCS form.
%Equivalently, $\Omega$ is LCS if $M$ can be covered by open sets $U_j$ on which there exist $\phi_j\in\ce(U_j,\RR)$ such that $\e^{-\phi_j}\Omega$ is symplectic form on $U_j$. The Lee form is then given by $\theta|_{U_j}=d\phi_j$. 
\end{definition}

If $\Omega$ is LCS on $M$ with Lee form $\theta$, then for any $f\in\ce(M,\RR)$ also $\e^f\Omega$ is LCS , with Lee form $\theta+df$. We denote by $[\Omega]=\{\e^f\Omega|f\in\ce(M,\RR)\}$ the corresponding conformal class, and we call $[\Omega]$ an \textit{LCS structure}. 

Denote by $\pi:\hat M\rightarrow M$ the minimal cover of $M$ on which $\pi^*\theta$ becomes exact. If we consider the period map corresponding to the de Rham class $[\theta]_{dR}$:
\begin{equation*}
\chi_{[\theta]}:\pi_1(M)\rightarrow \RR, \ \ \ \gamma\mapsto\int_\gamma\theta
\end{equation*}
then $\Gamma:=\pi_1(M)/\ker\chi_{[\theta]}$ is the deck group of $\hat M$. Let $\phi\in\ce(\hat M,\RR)$ be so that $\pi^*\theta=d\phi$. Then we have $\gamma^*\phi=\phi+\chi_{[\theta]}(\gamma)$, for any $\gamma\in\Gamma$. The form $\omega:=\e^{-\phi}\pi^*\Omega$ is a global symplectic form on $\hat M$ on which $\Gamma$ acts by strict homotheties, and $(\hat M,\omega)$ is called the minimal symplectic cover of the LCS manifold $(M,[\Omega],[\theta]_{dR})$.

Let us now suppose that there exists an integrable complex structure $J$ on $M$.

\begin{definition}
An LCS form $\Omega$ which is compatible with $J$, in the sense that $g(\cdot,\cdot):=\Omega(\cdot,J\cdot)$ is a $J$-invariant Riemannian metric on $M$, is called a \textit{locally conformally \Ka} (LCK) form (or metric). 
\end{definition}

In LCS geometry, one is naturally led to consider the following operator, called the twisted differential:
\begin{equation*}
d_\theta:\Omega^k(M)\rightarrow \Omega^{k+1}(M), \ \ \al\mapsto d\al-\theta\wedge\al.
\end{equation*}
Here $\theta$ can be any closed one-form, so that one would have $d_\theta^2=0$. The corresponding cohomology:
\begin{equation*}
H^\bullet_{\theta}(M)=\frac{\ker d_\theta}{\im d_\theta}
\end{equation*}
is called the Lichnerowicz or twisted de Rham cohomology. Equivalently, $d_\theta$ can be thought of as a flat connection $\nabla$ on the trivial real line bundle $M\times \RR$.  Such a line bundle with connection will then be denoted by $L_\theta$. If $d^\nabla$ denotes the differential operator acting on $\Omega^\bullet (M)\otimes L_\theta$ which is induced by $\nabla$ by the Leibniz rule, then again as $\nabla$ is flat, $(d^\nabla)^2=0$ and we have a natural isomorphism:
\begin{equation*}
H_\theta^\bullet(M)\cong H^\bullet(M,L_\theta):=\frac{\ker d^\nabla}{\im d^\nabla}.
\end{equation*}
Note that this cohomology only depends on the de Rham class $[\theta]_{dR}$.

If one additionally has an integrable complex structure $J$ on $M$, then one can split the twisted differential as $d_\theta=\del_\theta+\cp_\theta$, where:
 \begin{align*}
\del_\theta &:\Omega^{p,q}(M,\CC)\rightarrow \Omega^{p+1,q}(M,\CC), \ \  \del_{\theta}\al=\del\al-\theta^{1,0}\wedge\al, \\ 
\cp_\theta &:\Omega^{p,q}(M,\CC)\rightarrow \Omega^{p,q+1}(M,\CC), \  \ \cp_{\theta}\al=\cp\al-\theta^{0,1}\wedge\al.
\end{align*}
Here we let:
\begin{equation*}
\theta^{1,0}=\frac{1}{2}(\theta+iJ\theta)\in\Omega^{1,0}(M,\CC), \ \ \theta^{0,1}=\frac{1}{2}(\theta-iJ\theta)\in\Omega^{0,1}(M,\CC).
\end{equation*}
Clearly one has $\del_\theta^2=\cp_\theta^2=0$, and we will denote the twisted Dolbeault cohomology groups corresponding to $\cp_\theta$ by:
\begin{equation*}
H^{p,q}(M,L_\theta)=\frac{\ker \cp_\theta|_{\Omega^{p,q}(M,\CC)}}{\im \cp_\theta|_{\Omega^{p,q-1}(M,\CC)}}.
\end{equation*}

By definition, an LCS form verifies $d_\theta\Omega=0$, and so it induces a cohomology class $[\Omega]_{d_\theta}$ in $H^2_\theta(M)$. However this class can vanish, and this gives rise to particular kinds of LCK metrics:

\begin{definition}\label{DefPot} An LCK (or LCS) form $\Omega$ is called \textit{exact} if $\Omega=d_\theta\be$ for some $\be\in\Omega^1(M,\RR)$. 
An LCK metric $\Omega$ is called \text{with potential} if there exists $f\in\ce(M,\RR)$ so that $\Omega=2i\del_\theta\cp_\theta f$. It is called \textit{with positive potential} if $f$ can be chosen strictly positive on $M$.
\end{definition}

Note that the above definitions are conformally invariant, since we have:
\begin{align}
\e^ud_\theta\be&=d_{\theta+du}(\e^u\be), \nonumber \\
\e^u2i\del_{\theta}\cp_\theta f&=  2i\del_{\theta+du}\cp_{\theta+du}(\e^uf), \ \ \forall u,f\in\ce(M), \ \be\in\Omega^{1}(M).\label{ideldelconf}
\end{align}
Hence it make sense to say that a conformal LCK structure $[\Omega]$ is exact or with (positive) potential.

\subsection{Vaisman type manifolds}

A very special class of LCK metrics is given by the Vaisman metrics, which are particular examples of LCK metrics with positive potential. 
\begin{definition}\label{DefVais}
A strict LCK metric $\Omega$ on $(M,J)$ is called Vaisman if its Lee form $\theta$ is parallel with respect to the Levi-Civita connection corresponding to $g=\Omega(\cdot,J\cdot)$. A complex manifold $(M,J)$ is called \textit{of Vaisman type} if it admits some compatible Vaisman metric. 
\end{definition}

Note that for any LCS form $(\Omega,\theta)$, one can naturally define a vector field $A$, which we will call \textit{the anti-Lee} vector field, by:
\begin{equation}\label{antiLee}
\iota_A\Omega=-\theta.
\end{equation}
If moreover $\Omega$ is LCK with respect to $J$, then we can also define \textit{the Lee vector field} $B$ by $B=-JA$, so that $B$ is the metric dual of $\theta$.

It is not difficult to see that for a Vaisman metric, the Lee vector field $B$ is real holomorphic, Killing and of constant norm, while the Lee form is harmonic with respect to the Vaisman metric. On the other hand, if $n>1$, then up to constant multiples there exists at most one Vaisman metric in a given conformal LCK class, so this notion is not conformally invariant. We will usually normalise a Vaisman metric so that its Lee vector field is of norm $1$.

Also it is easy to check that given a Vaisman metric $(\Omega,\theta)$ with $\theta(B)=1=\Omega(B,JB)$, then because $\LL_B\Omega=0$, we have $\Omega=-dJ\theta+\theta\wedge J\theta$. This last equation is equivalent to $\Omega=2i\del_\theta\cp_\theta 1$, i.e. $\Omega$ has positive potential $f=1$.

We recall here a few properties of Vaisman type manifolds that will be used later in this paper. Fix a Vaisman type manifold $(M,J)$. Also let $B$ be the Lee vector field of some Vaisman metric on $(M,J)$.

\begin{fact}(\cite[Corollary~2.7]{t97} \label{FLeev}
The Lee vector field of any Vaisman metric on $(M,J)$ is a positive multiple of $B$. 
\end{fact}

\begin{fact}(\textnormal{Proof of} \cite[Theorem~5.1]{t94})\label{FLeef}
Given any LCK form with Lee form $\theta$ on $(M,J)$, there exists $\theta_0\in[\theta]_{dR}$ so that $\Omega_0:=2i\del_{\theta_0}\cp_{\theta_0}1$ is Vaisman on $(M,J)$.
\end{fact}

\begin{fact}(\cite[Theorem~4.5]{llmp} \textnormal{together with}  \cite[Theorem~5.1]{t94})\label{Fexact} For any $0\neq[\theta]_{dR}$ in $H^1(M,\RR)$ we have $H^\bullet_\theta(M)=0$. In particular, any LCK metric on $(M,J)$ is exact.
\end{fact}

\section{Toric LCS manifolds}\label{SecToricLCS}

\begin{definition}
A $2n$-dimensional LCS manifold $(M,[\Omega],[\theta]_{dR})$ endowed with an effective action of an $n$-dimensional torus $G=\TT^n$ is called \textit{a toric LCS manifold} if every vector field in the Lie algebra of the torus $\Lg=\Lie(G)\subset\ce(TM)$ is twisted Hamiltonian with respect to $[\Omega]$. This means that if we take any $\Omega\in[\Omega]$ with Lee form $\theta$ and for any $V\in\Lg$, there exists a function $\mu^\Omega_V\in\ce(M)$ so that:
\begin{equation}\label{Ham}
\iota_V\Omega=d_\theta \mu^\Omega_V.
\end{equation}
\end{definition}

For any other conformal LCS form $\Omega'=\e^f\Omega$ with Lee form $\theta'=\theta+df$, one has:
\begin{equation*}
\iota_V\Omega'=d_{\theta'}(\e^f \mu^\Omega_V)
\end{equation*}
so that the notion of a twisted Hamiltonian action is indeed conformally invariant. Moreover, if the LCS structure is strict, meaning that $[\theta]_{dR}\neq 0$ in $H^1(M,\RR)$, then $H^0_{\theta}(M)=0$ (cf. \cite[Proposition~2.1]{v85}) and thus the Hamiltonian $\mu^\Omega_X$ corresponding to $\Omega$ is uniquely defined by \eqref{Ham}. In particular, for every $\Omega \in[\Omega]$ we have a well determined \textit{moment map} $\mu^\Omega : M \rightarrow \Lg^*$ given by:
\begin{equation}
\langle\mu^\Omega,V\rangle=\mu^\Omega_V, \ \ \ V\in\Lg.
\end{equation}
Under conformal changes of the LCS form, it transforms as
\begin{equation}\label{confmu} 
\mu^{\e^f\Omega}=\e^f\mu^\Omega.
\end{equation}

One can show (cf. \cite[Corollary 4.5]{i17}) that if $(M,[\Omega],G)$ is a toric LCS manifold, then the action of $G$ lifts to the minimal cover, so that $(\hat M, \omega, G)$ becomes a toric symplectic manifold. If $(\Omega,\theta)$ is an LCS representative with moment map $\mu_\Omega$ and $\pi^*\theta=d\phi$ on $\hat M$, then the moment map of the symplectic cover is given by $\hat\mu=\e^{-\phi}\pi^*\mu_{\Omega}$. 

Note that, by \eqref{confmu}, the image of the moment map of an LCS form $\Omega $ is not invariant under conformal changes of the form. However, the cone over it
\begin{equation*} 
C:=\RR_{\geq 0}\cdot\mu_{\Omega}(M)=\{t\mu_\Omega(x)|t\geq 0, x\in M\}=\RR_{\geq 0}\cdot\hat\mu(\hat M)\subset\Lg^*\end{equation*}
is an invariant of the conformal class $[\Omega]$, and will be called \textit{the moment cone of the LCS structure}.

\begin{definition}
We say that a compact toric LCS manifold $(M, [\Omega], [\theta]_{dR}, G)$ with $[\theta]_{dR}\neq 0$ is \textit{good} if there exists an integrable complex structure $J$ on $M$ which is compatible with all the data. More precisely, this means that $([\Omega], J)$ is an LCK structure on $M$ and that $G$ acts by biholomorphisms with respect to $J$, so that $(M,J,[\Omega],[\theta]_{dR},G)$ is a toric LCK manifold.
\end{definition}

\begin{remark}
This definition is given in analogy with Lerman's definition \cite{l03} of \textit{good cones} (cf. \ref{goodC}). As will be apparent later (cf. \ref{Vaismant}), there exists an actual correspondence between the two classes of \textit{good} objects. 
\end{remark}

\begin{remark}
When $[\theta]_{dR}=0$, it follows a posteriori from the Delzant construction that $(M,[\Omega], G)$ admits a compatible complex structure. However, this no longer holds in the strict LCS case, as is shown by \cite[Example~6.3]{i17}. 
\end{remark}

\section{Toric symplectic cones }\label{SecSymplCone}

In this section we give a brief presentation of the toric symplectic cones. After giving the main definitions, we recall the combinatorial classification of a certain subclass of these manifolds, called \textit{good toric symplectic cones},  started by Banyaga-Molino, Boyer-Galicki and achieved by Lerman. This subclass is precisely the one we have to deal with in order to classify good toric LCS manifolds. For details about this section, one can check \cite{l03} or \cite{l03b}.

\begin{definition}
A \textit{symplectic cone} is a connected symplectic manifold $(N,\omega)$ endowed with a vector field $X\in\ce(TN)$, called \textit{the Liouville vector field}, which generates a proper $\RR$-action $(\rho_t)_t\subset\Aut(N)$ by contractions of $\omega$: 
\begin{equation}\label{dil}
\LL_X\omega=-\omega, \ \ \text{ or equivalently  \ } \rho_t^*\omega=\e^{-t}\omega, \ \ t\in\RR.
\end{equation}
We denote by $\la=-\iota_X\omega$ the \textit{Liouville form}, so that $\omega=d\la$.
\end{definition}

Note that \eqref{dil} implies that the $\RR$-action is effective on $N$. As the action is moreover proper, and as $\RR$ has no non-trivial compact subgroup which would constitute the eventual stabiliser of some point, the action is then free and one has a smooth quotient $S=N/\RR$. If $S$ is compact, then $(N,\omega, X)$ is called \textit{of compact type}.

\begin{remark}\label{contact}
Given a symplectic cone $(N,\omega, X)$ with $S=N/\RR$, the natural projection $p:N\rightarrow S$ is an $\RR$-principal bundle and $S$ is naturally endowed with a co-oriented contact structure given by the uniquely defined conformal class: \begin{equation}\label{contactstr}
[\al]=\{\al\in\ce(T^*S)|\ p^*\al=\e^f\la, \ f\in\ce(N)\}.
\end{equation}
Conversely, the symplectisation of any co-oriented contact manifold has a natural structure of a symplectic cone (see for instance \cite[Chapter~2]{l03b}). 
\end{remark}

\begin{remark}\label{trivialisation}
As $\RR$ is contractible, the principal bundle $p$ is trivial. Each choice of a contact form $\al\in[\al]$ corresponds to a trivialisation $F_\al:S\times \RR\rightarrow N$. Indeed, $\al$ determines a smooth function $f:N\rightarrow \RR$ such that $p^*\al=\e^f\la$. It can easily be seen that $f$ is $\RR$-equivariant, where $\RR$ acts on the co-domain of $f$ by translations. Then $S_N:=f^{-1}(0)\subset N$ is a slice of the $\RR$-action on $N$, $p|_{S_N}:S_N\cong S$ and one has
\begin{equation*}
F_\al(y,t)=\rho_{t}(p|_{S_N}^{-1}(y)), \ \ \ (y,t)\in S\times\RR.
\end{equation*}
Conversely, a trivialisation $F:S\times\RR\rightarrow N$ defines a contact form $\al\in[\al]$ by $\al:=\e^tF^*\la$. We then have $F^*X=\frac{\del}{\del t}$ and $F^*\omega=d(\e^{-t}F^*\al)$. 
\end{remark}

Given a symplectic cone $(N,\omega, X)$, we say that a complex structure $J$ on $N$ is \textit{compatible} if $(\omega, J)$ is a \Ka\ structure and $\LL_X J=0$. In this case, $(N,J,\omega,X)$ is called a \textit{\Ka\ cone}. A compatible complex structure $J$ determines a natural trivialisation of the principal bundle $p$. Indeed, $X$ has no zeroes as we already noted, hence we have a positive function $\e^{-f}:=\omega(X,JX)$. As both $X$ and $JX$ are $\RR$-invariant, it is clear that the function $f:N\rightarrow \RR$ is equivariant, and thus defines a trivialisation $F:S\times\RR\rightarrow N$ just as in \ref{trivialisation}. 

\begin{remark} Let $(M,J,\Omega,\theta)$ be a compact Vaisman manifold with Lee vector field $B$. Let $\pi:\hat M\rightarrow M$ be its minimal cover and let $\pi^*\theta=d\phi$. Then $(\hat M, J, \omega=\e^{-\phi}\pi^*\Omega, B)$ is a \Ka\ cone. (cf. \cite{v79} and \cite[Theorem~4.2]{gopp}) It is of compact type if and only if $[\theta]_{dR}$ is a multiple of an element from $H^1(M,\QQ)\subset H^1(M,\RR)$. Conversely, given a \Ka\ cone $(N=S\times \RR, J,\omega, X=\frac{\del}{\del t})$ together with a discrete group $\Gamma$ acting freely, properly and holomorphically on $(N,J)$, acting by homotheties on $\omega$ and preserving $X$, the quotient $(N/\Gamma, J, \e^{t}\omega, dt)$ is a Vaisman manifold with Lee vector field $X=B$ (cf. \cite[Proposition~7.3]{gop}). 
\end{remark}

\begin{definition}
A \textit{toric symplectic cone} is a $2n$-dimensional symplectic cone $(N,\omega, X)$ endowed with an effective symplectic action of a torus $G=\TT^n$ which preserves the Liouville field $X$, and with a \emph{moment map} $\mu:N\rightarrow \Lg^*$ verifying the equivariance condition:
\begin{equation*}
\LL_X\mu=-\mu.
\end{equation*}
The set $C=\mu(N)\cup\{0\}\subset\Lg^*$ is called \textit{the moment cone} of $(N,\omega, X,\mu)$. If $(N,\omega,X,\mu)$ is moreover of compact type and admits a $G$-invariant compatible complex structure, then it is called a \textit{good} toric symplectic cone. 
\end{definition}

\begin{remark}\label{contactToric}
Given a toric symplectic cone $(N,\omega, X,G)$, there exists a natural effective action of $G$ on $S=N/\RR$ which makes the projection $p:N\rightarrow S$ $G$-equivariant. It is easy to check that $G$ preserves the contact structure $[\al]$ defined by \eqref{contactstr}, since it preserves $\omega$ when acting on $N$.
\end{remark}

Good symplectic cones are classified by their moment cones, which are polyhedral cones in $\Lg^*$ with certain combinatorial properties. Let $G=\Lg/\Lambda$ be a compact $n$-dimensional torus with Lie algebra $\Lg$ and integral lattice $\Lambda=\ker(\exp:\Lg\rightarrow G)$. We denote by $\Lg^*$ the dual Lie algebra. 

\begin{definition}(\cite{l03})\label{goodC}
Let $C\subset \Lg^*$ be a \textit{rational polyhedral cone}. This means that there exists a minimal a set of primitive vectors $\nu_1,\ldots,  \nu_d\in\Lambda$, $d\geq n$, defining $C$:
\begin{equation}\label{defC}
C=\{ l\in\Lg^*| \langle l,\nu_j\rangle \geq 0, j=1,d\}.
\end{equation}
%We can, and will, suppose that each vector $\nu_j$ is primitive and that the family $\{\nu_1,\ldots, \nu_d\}$ is minimal for \eqref{defC}. 
A subset of the form
\begin{equation*}
F_j=C\cap\{l\in\Lg^*| \langle l, \nu_j\rangle=0\}\subset C
\end{equation*}
is called a facet of $C$ and $\nu_j$ is its defining normal. 

The cone $C$ is called \textit{a good cone} if it has non-empty interior and if every $k$-codimensional face of $C$, $0<k<n$, is the intersection of exactly $k$ facets whose defining normals can be completed to a $\ZZ$-basis of $\Lambda$. This formalizes as follows: for every face $F$ of $C$
\begin{equation*}
F=C\cap \{ l\in\Lg^* |\langle l, \nu_{j_1}\rangle=0, \ldots, \langle l, \nu_{j_k}\rangle=0 \}\neq\emptyset
\end{equation*}
the corresponding annihilator $\Lg^\circ_F:=\spn_{\RR}\{\nu_{j_1},\ldots, \nu_{j_k}\}\subset \Lg$ is of dimension $k$ and 
\begin{equation*}
\Lambda_F:=\Lg^\circ_F\cap\Lambda=\textstyle{\spn_{\ZZ}} \{ \nu_{j_1},\ldots, \nu_{j_k}\}
\end{equation*}
is a $k$-dimensional lattice. Thus, each $k$-codimensional face $F$ gives rise to a $k$-dimensional subtorus $G_F:=\Lg_F^\circ/\Lambda_F\subset G$.
\end{definition}

\begin{theorem}\textnormal{(Banyaga-Molino, Boyer-Galicki, Lerman)}\label{classif}
A toric symplectic cone of compact type is good if and only if its moment cone is good. Moreover, for each good cone $C\subset\Lg^*$ there exists a unique good symplectic cone $(N_C,\omega_C,X_C,\mu_C)$ with moment cone $C$. 
\end{theorem}

\section{Classification of toric LCS manifolds of LCK type}\label{SecClass}

In this section we intend to give a combinatorial classification of compact toric LCS manifolds which admit a compatible complex structure, in the spirit of \ref{classif}. In order to do so, we will first need to recall, without proofs, the main steps of the construction of toric Vaisman metrics given in \cite{i17}. Next, we derive further consequences of this construction, until we are finally able to use \ref{classif} in order to show our result.

 %    using this construction, we are ab     We will moreover need to recall the main steps of the construction given in our paper, without proofs. , together with the main steps of the construction of the metric. This will be done without any proofs. We will describe the main steps in the proof 
The exact statement that we will prove is the following:

\begin{theorem}\label{classifLCS}
Let $G$ be a compact torus. There exists a one-to-one correspondence between: 
\begin{enumerate}[(a)]
\item good toric LCS $G$-manifolds up to $G$-equivariant conformal automorphisms
\item pairs $(C,a)$, where $C\subset\Lg^*$ is a good cone and $a\in\RR_{>0}$. 
\end{enumerate}
\end{theorem}

Let us start by recalling the result of \cite{i17} together with a sketch of its proof:

\begin{theorem}(\cite[Theorem~A]{i17})\label{toricLCK}
Let $(M, J, [\Omega],[\theta]_{dR},G)$ be a compact toric LCK manifold with $[\theta]_{dR}\neq 0$. Then there exists an LCS form $\Omega_1$ so that $(M, J, \Omega_1, G)$ is a toric Vaisman manifold.
\end{theorem}

The Vaisman form $\Omega_1$ is constructed as follows. We first choose a $G$-invariant representative $(\Omega,\theta)$ in $[\Omega]$, and let $(\hat M, \omega=\e^{-\phi}\pi^*\Omega,G)$ be the corresponding toric symplectic minimal cover, where $\pi^*\theta=d\phi$. We denote also by $J$ the pull-back complex structure on $\hat M$. 

The complex structure $J$ determines a complexified torus $G^J\cong(\CC^*)^n$ which acts effectively on $\hat M$, so that $G^J \subset\Aut(\hat M)$. One shows then that the deck group $\Gamma$ of the covering $\hat M\rightarrow M$ is free abelian of rank $1$ and is a subgroup of $G^J$. Let $\gamma$ be the generator of $\Gamma$ on which $\int_\gamma\theta=a >0$. Then there exists an element:
\begin{equation}\label{d}
 d\in\Lie(G^J)=\Lg\oplus J\Lg, \ \  d=d_\Lg+d_{J\Lg} \ \text{ with } d_{J\Lg}\neq 0,
\end{equation} 
so that $\gamma=\exp_{G^J} d$. Thus we have an effective $\RR$-action on $\hat M$ given by $\Phi_t(x)=\exp_{G^J}(td).x$ which commutes with the $G$-action and descends to an $\Ss^1=\RR/\ZZ$-action on $M$, holomorphic with respect to $J$. We denote the corresponding one-parameter group on $M$ also by $\Phi_t$. %Let us also denote by $H_J:=G\times \Ss^1\subset\Aut(M)$ the resulting $(n+1)$-torus. Note that, as $d$ is uniquely defined modulo $\Lambda\subset\Lg$, the torus $H_J$ only depends on the complex structure $J$. %(see further \ref{torHJ}).  

Next, one averages $\theta$ over $\Ss^1$ in order to obtain an invariant form: \begin{equation*}
\theta_0=\int_0^1\Phi_s^*\theta ds=\theta+du, \ \ \ u\in\ce(M) \ G\text{-invariant}. 
\end{equation*}
Let us define:
\begin{equation}\label{Ot}
\Omega_t=\frac{1}{t}\int_0^t\Phi_s^*(\e^u\Omega)ds, \ \ \ t\in(0,1], \ \ \Omega_0=\e^u\Omega.
\end{equation}
Then $(\Omega_t)_{t\in [0,1]}$ is a smooth family of $G$-invariant LCS forms on $M$ with Lee form $\theta_0$. As it turns out, $\Omega_1$ is a  toric Vaisman structure with respect to $J$. 

\begin{remark}\label{theLee}
Note that in the formula \eqref{d} defining the vector field $d$, we have $d_{J\Lg}=d-d_{\Lg}\in\aut(J,\Omega_1)$, $Jd_{J\Lg}\in\Lg\subset\aut(J,\Omega_1)$ and moreover $\theta_0(d_{J\Lg})=\theta_0(d)=a\neq 0$. It follows thus, by \cite[Proposition~3]{i18}, that $d_{J\Lg}$ is a constant multiple of the Lee vector field of the Vaisman metric. 
\end{remark}

\begin{corollary}\label{Vaismant}
Any good compact toric LCS manifold $(M, [\Omega], [\theta]_{dR}, G)$ is of Vaisman type, meaning that there exists a complex structure $J'$ and a representative $\Omega'\in[\Omega]$ so that $(\Omega', J')$ is a $G$-invariant Vaisman structure on $M$. Thus its minimal cover $(\hat M,\omega, G)$ admits the structure of a good toric symplectic cone. Moreover, there exists a co-oriented contact manifold $(S,\al)$ endowed with a $G$-action for which $\al$ is $G$-invariant and a $G$-equivariant diffeomorphism $F:M\rightarrow S\times \Ss^1$, so that:
\begin{equation}\label{contactLCS}
F^*(d\al-dt\wedge\al)=\Omega'
\end{equation}
where $t$ denotes the local coordinate on $\Ss^1$.
\end{corollary}

\begin{proof}

We fix a compatible complex structure $J$, and use the same notation as before. \ref{toricLCK} together with \ref{Fexact} imply that $\Omega$ and each $\Omega_t$ defined by \eqref{Ot} are exact LCS forms. We are thus in the hypotheses of the LCS variant of the Moser stability theorem, namely: we have a family of LCS forms $\Omega_t$ with constant de Rham Lee class $[\theta]_{dR}$ and with constant class in the Lichnerowicz  cohomology $[\Omega_t]_{d_\theta}=0\in H^2_{\theta}(M,\RR)$. Then Moser's trick (\cite{b02} and \cite{bk}) ensures the existence of an isotopy $(\psi_t)_{t\in [0,1]}\subset\Aut(M)$, $\psi_0=\id_M$, with $\Omega':=\psi_1^*\Omega_1\in[\Omega_0]=[\Omega]$. Thus:
 \begin{equation*}
 J':=\psi_1^*J=d\psi_1^{-1}Jd\psi_1
 \end{equation*} 
 is an integrable complex structure on $M$ compatible with $[\Omega]$ and $(\Omega', J')$ is Vaisman.

In order to see that $\Omega'$ and $J'$ are moreover $G$-invariant, we need to recall the construction of $\psi_t$ and to check that it is $G$-equivariant. Let us thus choose $\eta\in\ce(T^*M)$ $G$-invariant so that $\Omega_0=d_{\theta_0}\eta$. Note that, as $\theta_0$ is already $G$-invariant, this is always possible after averaging over $G$ any $d_{\theta_0}$ primitive of $\Omega_0$.  

Define the smooth family of $G$-invariant one forms:
\begin{equation*}
\eta_t:=\frac{1}{t}\int_0^t\Phi_s^*\eta ds, \ \ \ t\in (0,1], \ \ \eta_0=\eta
\end{equation*}
so that $\Omega_t=d_{\theta_0}\eta_t$. Let $X_t$ be the time-dependent $G$-invariant vector field given by:
\begin{equation*}
\iota_{X_t}\Omega_t=-\frac{d}{dt}\eta_t.
\end{equation*}
Then $X_t$ uniquely defines a family of diffeomorphisms $\psi_t$ by:
\begin{equation}\label{phit}
\frac{d}{dt}\psi_t(x)=X_t(\psi_t(x)), \ \ \psi_0(x)=x, \ \ x\in M
\end{equation}
and one checks that $\psi_t$ acts on $\Omega_t$ by:
\begin{equation*}
\psi_t^*\Omega_t=\e^{\int_0^t\psi_s^*\theta_0(X_s)ds}\Omega_0.
\end{equation*}

Note that as $X_t$ is $G$-invariant, by the uniqueness of the solution of \eqref{phit}, $\psi_t$ is indeed $G$-equivariant.

Denote by $\theta'=\psi_1^*\theta_0$ the Lee form of $\Omega'$ and let $\pi^*\theta'=d\phi'$ on $\hat M$. Then the vector field $X:=\frac{1}{a}(\psi_1^{-1})_*d$ is a Liouville vector field for the minimal symplectic cover $(\hat M, \omega=\e^{-\phi'}\pi^*\Omega',G)$, as $\theta'(X)=\frac{1}{a}\theta_0(d)=1$ and we have:
\begin{equation*}
\LL_{aX}\Omega'=\psi_1^*(\LL_d\Omega_1)=0.
\end{equation*} 
Also, since $\Gamma$ is of rank one, $\hat M/\RR$ is compact, so $(\hat M,\omega, X,G)$ is a good toric symplectic cone. 

Finally, note that the vector field $X$ together with the $G$-invariant function $\phi'$ define a $G$-equivariant diffeomorphism as in \ref{trivialisation}:
\begin{equation*}
S\times\RR\rightarrow \hat M, \ \ (y,t)\mapsto \Phi_X^t(y)
\end{equation*}
where $S=(\phi')^{-1}(0)$ is endowed with the induced $G$-action from $\hat M$. This diffeomorphism descends to a diffeomorphism $F:S\times \RR/a\ZZ\rightarrow M$ so that $F^*\Omega'=d\al-dt\wedge \al$, where $\al=-\iota_X\omega|_S$ is a $G$-invariant contact form on $S$. 
\end{proof}

\begin{proof}[\textbf{Proof of~\ref{classifLCS}}]
The proof consists in the following two claims:

\begin{clm} One can naturally associate to a good toric LCS manifold $(M,[\Omega],[\theta]_{dR}, G)$ a pair consisting in a good cone $C\subset\Lg^*$ and a positive number $a$. Moreover, the pair $(C,a)$ is invariant to conformal $G$-equivariant automorphisms of $(M,[\Omega], [\theta]_{dR}, G)$. 
\end{clm}

The real $a>0$ is defined to be the first positive period of $[\theta]_{dR}$. More precisely, if $\Gamma$ denotes the deck group of the minimal cover $(\hat M,\omega)$, we have the period morphism 
\begin{equation*}
\chi_{[\theta]}:\Gamma\rightarrow \RR, \ \ \gamma\mapsto \int_\gamma\theta.
\end{equation*}
As $\Gamma$ is of rank one, there exists a unique number $a>0$ so that $\im\chi_{[\theta]}=a\ZZ\subset\RR$.
 
We let $C$ be the moment cone of the LCS manifold. By \ref{Vaismant}, there exists a vector field $X\in\ce(T\hat M)$ so that $(\hat M,\omega, X,G)$ becomes a good symplectic cone. By \ref{classifLCS}, $C$, which is also the cone of the symplectic cover, is a good cone.

Finally, suppose that we have a $G$-equivariant automorphism $F:(M_1,[\Omega_1],[\theta_1]_{dR},G)\rightarrow (M_2,[\Omega_2],[\theta_2]_{dR},G)$. Fix LCS forms $\Omega_1\in[\Omega_1]$ and $\Omega_2\in[\Omega_2]$ so that $F^*\Omega_2=\Omega_1$. Then clearly $\mu_{\Omega_1}=\mu_{\Omega_2}\circ F$, so the moment cones of the two manifolds coincide. Moreover, we have $F^*[\theta_2]_{dR}=[\theta_1]_{dR}$, so $\chi_{[\theta_1]}=\chi_{[\theta_2]}\circ F_*$, where $F_*$ is the morphism induced at the level of deck groups. Thus we have an equality between the corresponding periods $a_1=a_2$.

\begin{clm}
Given a pair consisting in a good cone $C\subset\Lg^*$ and a positive number $a$, there exists a good toric LCS manifold $(M_{C,a},[\Omega_{C,a}],G)$ associated to it, which is unique up to equivariant automorphisms.
\end{clm}

By \ref{classif}, there exists, up to $G$-equivariant automorphisms, a unique good symplectic cone $(N_C,\omega_C, X_C,G)$ having $C$ as moment cone. Let $(S,[\al])$ be the corresponding contact manifold, cf. \ref{contact}, endowed with the natural action of $G$, cf. \ref{contactToric}. After choosing a $G$-invariant contact form $\al\in[\al]$, we can identify the symplectic cone with $(S\times \RR, d(\e^{-t}\al), \frac{d}{dt}, G)$, cf. \ref{trivialisation}, where $G$ acts trivially on the $\RR$ factor. Then $M_{C,a}:=S\times\RR/a\ZZ$ endowed with the induced $G$-action and with the LCS structure $[\Omega_{C,a}]=[d\al-dt\wedge\al]$ is a good toric LCS manifold with moment cone $C$.

Suppose now we are given two good toric LCS $G$-manifolds $(M_j, [\Omega_j], G)$, $j=1,2$, with moment cones $C_1=C_2=C$ and corresponding periods $a_1=a_2=a$. Let $X_1$ and $X_2$ be compatible Liouville vector fields on the minimal covers $(\hat M_j,\omega_j,G)$ and let $(S_j=\hat M_j/\RR,[\al_j],G)$ be the corresponding contact manifolds endowed with the natural actions of $G$. By \ref{classif}, there exists a $G$-equivariant isomorphism $\hat F$ between the two symplectic cones. As $\hat F_*X_1=X_2$, it induces a $G$-equivariant contactomorphism $F:(S_1,[\al_1],G)\rightarrow (S_2,[\al_2],G)$. Let $\al_1=F^*\al_2$. 

On the other hand, by \ref{Vaismant}, we have $(M_j,[\Omega_j],G)\cong (S_j\times\RR/a\ZZ, [d\al_j-dt\wedge \al_j],G)$, $j=1,2$, hence $F\times\id_{\RR/a\ZZ}$ induces the desired automorphism between the two LCS manifolds. 
\end{proof}

\subsection*{The image of the moment map}

Consider a good toric LCS manifold $(M,[\Omega],[\theta]_{dR},G)$ with moment cone $C$ and generating period $a>0$. Let $(\hat M,\omega,G)$ be the minimal symplectic cover. Let us moreover fix a compatible Vaisman structure $(\Omega=-dJ\theta+\theta\wedge J\theta, J)$, and let $\mu=\mu_\Omega:M\rightarrow \Lg^*$ be the corresponding moment map. The anti-Lee vector field $A$ defined by $\iota_A\Omega=-\theta$ belongs to $\Lg$ (\cite[Lemma~4.7]{p16}) and has corresponding Hamiltonian $\mu_A=1$. Thus, letting $H_A\subset\Lg^*$ be the affine hyperlpane defined by $A$:
\begin{equation*}
H_A=:\{l\in\Lg^*| \langle l,A\rangle=1\}\subset\Lg^*
\end{equation*}
we have $P_A:=\im\mu(M)=C\cap H_A$. In particular, $P_A$ is an $(n-1)$-dimensional convex polytope. 
%From the definition of good cones, it follows easily that $P_A$ is a \textit{simple polytope}, i.e. at each of its vertex there are exactly $n-1$ edges that meet. Moreover, \tilde{P_A}=\{l\in C| \langle l,A\rangle \geq 1\} is a $n$-dimensional simple convex polytope in $\Lg^*$. All of its facets but the one defined by $A$ are rational, so .. 

Note that, unlike in other toric geometries, $\mu(M)$ does not represent the orbit space of the action of $G$ on $M$. In fact, if we fix an isomorphism $(M,\Omega,G)\cong(S\times\Ss^1, d_{dt}\al,G)$, then the fiber of $l\in P_A$ is given by $\mu^{-1}(l)=\Oo_l\times\Ss^1$, where $\Oo_l$ denotes a $G$-orbit of $G$ on $S$, and thus consists in a circle of $G$-orbits of $M$. In order to understand best the orbit space $M/G$, let us note that, by \ref{classif} and \cite[Lemma~4.3]{l03}, $\hat\mu$ induces a homeomorphism from the orbit space $\hat M/G$ to $\hat\mu(\hat M)=C-\{0\}=:C^*$. Moreover, we have a well-defined free action of $\Gamma$ on $C^*$, with respect to which $\hat{\mu}$ is $\Gamma$-equivariant:
\begin{equation*}
\gamma.l=\e^{-\chi_{[\theta]}(\gamma)}l, \ \ l\in C^*.
\end{equation*}
We thus infer that $\hat\mu ($mod $\Gamma):M\rightarrow C^*/ \Gamma$ is well-defined and induces a homeomorphism: 
\begin{equation*}
M/G\cong C^*/\Gamma\cong P_A\times\RR/a\ZZ.
\end{equation*}
In particular, the orbit space is not contractible.

\section{LCK metrics with potential}\label{SecPot}

This section is dedicated to the proof of the following result:

\begin{theorem}\label{toricPot}
Let $(M,J,[\Omega],[\theta]_{dR},G)$ be a compact toric LCK manifold. Then any $G$-invariant representative $\Omega\in[\Omega]$ admits a unique $G$-invariant potential. Moreover, the potential is positive.
\end{theorem}

If we weaken the hypothesis by imposing less symmetry on the LCK metric, then we can still arrive at the same conclusion but without the positivity of the potential, namely:

\begin{theorem}\label{APot}
Let $(M,J)$ be a compact complex manifold of Vaisman type and let $A$ be the anti-Lee vector field of some Vaisman metric. Let $([\Omega],[\theta]_{dR})$ be an LCK structure on $(M,J)$. If $A\in\aut([\Omega])$, then the LCK structure $[\Omega]$ admits a potential. Moreover, for each $A$-invariant form $\Omega\in[\Omega]$, there exists a unique $A$-invariant potential for $\Omega$.
\end{theorem}

Recall from Section~\ref{SecDef} that an LCK metric $(\Omega,\theta)$ on $(M,J)$ is said to admit a potential if there exists a real function $f\in\ce(M,\RR)$ so that $\Omega=2i\del_{\theta}\cp_{\theta}f$. More generally, one can define the twisted Bott-Chern cohomology group:
\begin{equation}\label{BC}
H_{BC}^{1,1}(M,L_{\theta}):=\frac{\{\al\in\Omega^{1,1}(M,\RR)| d_\theta\al=0\}}{\{i\del_\theta \cp_\theta f|f\in\ce(M,\RR)\}}.
\end{equation}
Because of the relation \eqref{ideldelconf}, this cohomology group only depends on the complex structure and on de Rham class of $\theta$. As such, any LCK structure $[\Omega]$ on $(M,J)$ determines a twisted Bott-Chern cohomology class $[\Omega]_{BC}\in H_{BC}^{1,1}(M,L_{\theta})$, and by definition $[\Omega]$ admits a potential if and only if $[\Omega]_{BC}=0$.

In order to describe the twisted Bott-Chern cohomology group, recall that we have introduced in Section~\ref{SecDef} the groups $H^\bullet(M,L_\theta)$ and $H^{\bullet,\bullet}(M,L_\theta)$. Let us moreover define the following cohomology group, determined also by $[\theta]_{dR}$ and $J$:
\begin{equation*}
H^{1,1}_\RR(M,L_\theta):=\frac{\{\al\in\Omega^{1,1}(M,\RR)|d_\theta\al=0\}}{\{d_\theta\be|\be\in\Omega^1(M,\RR)\}}.
\end{equation*}
Note that $H^{1,1}_\RR(M,L_\theta)$ naturally identifies with a subgroup of $H^2(M,L_\theta)$. 

We can identify $H^1(M,L_\theta)$ with a subgroup of $H^{0,1}(M,L_\theta)$ via the injection:
\begin{equation*}
H^1(M,L_\theta)\rightarrow H^{0,1}(M,L_\theta), \ \ \al \bmod \im d_\theta\mapsto \al^{0,1} \bmod \im\cp_\theta.
\end{equation*}
Moreover, we have a morphism 
\begin{align*}
F:H^{0,1}(M,L_\theta)&\rightarrow H^{1,1}_{BC}(M,L_\theta)\\
\al\bmod\im \cp_\theta &\mapsto \Re\del_\theta\al\bmod\im(i\del_\theta\cp_\theta)
\end{align*}
which vanishes when restricted to $H^1(M,L_\theta)$, so that $F$ induces a morphism $[F]$ defined on the quotient $H^{0,1}(M,L_\theta)/H^1(M,L_\theta)$. The twisted Bott-Chern group is then described by the following exact sequence:
\begin{equation}\label{BCseq}
\xymatrix{
0\ar[r]&\frac{H^{0,1}(M,L_\theta)}{H^1(M,L_\theta)}\ar[r]^-{[F]}& H^{1,1}_{BC}(M,L_\theta)\ar[r]^-{[\id]} &H^{1,1}_\RR(M,L_\theta)\ar[r] &0.
}
\end{equation}
For the non-twisted version of this sequence, see \cite{g76}, and for the twisted one, see \cite{go14}.

Suppose now that $(M,J)$ is of Vaisman type. In this case, because $H^\bullet(M,L_\theta)=0$ by \ref{Fexact} and because of the exact sequence \eqref{BCseq}, $F$ gives rise to an isomorphism:
\begin{equation}\label{BCV}
H^{0,1}(M,L_\theta)\cong H^{1,1}_{BC}(M,L_\theta).
\end{equation}

Any LCK metric on a Vaisman type manifold $(M,J)$ is exact. However, it was observed by Goto \cite[Section~5.2]{go14} that $H^{0,1}(M,L_\theta)$ might not vanish, in which case there exist LCK metrics on $(M,J)$ which do not admit potentials. 

\begin{example}\label{exampleExact}
Consider the standard Hopf surface $H=\CC^{2}-\{0\}/_{z\sim \e z} $ with the Vaisman metric 
\begin{equation*}
\Omega_0=|z|^{-2}\sum idz_j\wedge d\ov{z}_j,\ \ \  \theta=-d\ln|z|^2.
\end{equation*} 
Let $\al=|z|^{-2}z_1^2\cp\ln|z|^2\in\ker\cp_\theta|_{\Omega^{0,1}(H,\CC)}$. Then for $K>0$ big enough, the following real $(1,1)$-form on $H$ is strictly positive:
\begin{equation*}
\Omega=K\Omega_0+\del_\theta\al+\cp_\theta\ov{\al}
\end{equation*}
defining an LCK metric $(\Omega,\theta)$ on $H$. However, it can be seen that the form $\al$ is not $\cp_\theta$-exact, so $\Omega$ admits no potential. 

Note that while $(H,\Omega_0,\theta)$ endowed with the standard $G=\TT^2$-action is a toric LCK manifold, the form $\al$ is not $G$-invariant and $([\Omega],[\theta]_{dR})$ is not a toric LCS structure.  
\end{example}

\begin{proof}[\textbf{Proof of \ref{APot}}]
Let $([\Omega],[\theta]_{dR})$ be an LCK structure on the Vaisman type manifold  $(M,J,A)$ on which $A$ acts conformally. By \ref{FLeef} and \ref{FLeev}, there exists a Vaisman metric $\Omega_0$ on $(M,J)$ with Lee form $\theta\in[\theta]_{dR}$ and with anti-Lee vector field $A$. Let $B=-JA$, and let us suppose, after eventually multiplying $\Omega_0$ and $A$ with positive constants, that $\theta(B)=|B|_{\Omega_0}^2=2$. 
 
We start by establishing some Hodge theoretical facts that we will need for our proof. For a differential operator $D$ on $M$, let us denote by $D^*$ its formal adjoint with respect to the Vaisman metric $g_0:=\Omega_0(\cdot,J\cdot)$. Let also:
\begin{equation*}
B^{0,1}=\frac{1}{2}(B+iA)\in\ce(T^{0,1}M), \ \ B^{1,0}=\frac{1}{2}(B-iA)\in\ce(T^{1,0}M)
\end{equation*}
so that $\theta^{1,0}(B^{1,0})=\theta^{0,1}(B^{0,1})=1$.  We have: 
\begin{equation*}
\cp_\theta^*=\cp^*-\iota_{B^{0,1}}.
\end{equation*}
Then the second order differential operator:
\begin{equation*}
\ov{\DA}_\theta=[\cp_\theta,\cp_\theta^*]=\cp_\theta\cp_\theta^*+\cp_\theta^*\cp_\theta
\end{equation*}
is an auto-adjoint generalized Laplacian, as its principal symbol is the same as the principal symbol of the $\cp$-Laplacian $\ov{\DA}=[\cp,\cp^*]$. Here and in what follows, for two graded operators $P$ and $Q$ on $\Omega^{\bullet}(M)$ we denote by $[P,Q]=PQ-(-1)^{\deg P\deg Q}QP$ the super-commutator of $P$ and $Q$, where $\deg$ denotes the degree of an operator.

Thus by Hodge theory, we have an $L^2$-orthogonal decomposition:
\begin{equation}\label{L2decomp}
\Omega^{\bullet,\bullet}(M)=\ker\ov{\DA}_\theta\oplus\im\cp_\theta\oplus\im\cp^*_\theta 
\end{equation} 
and an isomorphism:
\begin{equation}\label{Hodgeiso}
H^{0,1}(M,L_\theta)\cong\ker\ov{\DA}_\theta|_{\Omega^{0,1}(M)}.
\end{equation}

On the other hand, $A$ and $B$ are Killing vector fields for $g_0$ and their metric duals, $J\theta$ and $\theta$ respectively, are $d^*$-closed. This easily implies:
\begin{align*}
\LL_B^*=-\LL_B, & \ \ \LL_A^*=-\LL_A,\\
\LL^*_{B^{1,0}}=-\LL_{B^{0,1}}, & \ \  \LL^*_{B^{0,1}}=-\LL_{B^{1,0}}.
\end{align*}

As $B^{1,0}$ is a holomorphic vector field, we have $[\cp, \iota_{B^{1,0}}]=0$, hence also $[\del, \iota_{B^{0,1}}]=0$ and so:
\begin{equation*}
\LL_{B^{1,0}}=\del\iota_{B^{1,0}}+\iota_{B^{1,0}}\del, \ \ \LL_{B^{0,1}}=\cp\iota_{B^{0,1}}+\iota_{B^{0,1}}\cp.
\end{equation*}
Thus $\LL_{B^{0,1}}$ commutes with $\ov{\DA}_\theta$, implying $\LL_{B^{0,1}}(\ker\ov{\DA}_\theta)\subset\ker\ov{\DA}_\theta$. Also, as $\LL_{B^{1,0}}=-\LL_{B^{0,1}}^*$, we have $\LL_{B^{1,0}}(\ker\ov{\DA}_\theta)\subset\ker\ov{\DA}_\theta$.

Let now $\be\in\ker\ov{\DA}_\theta=\ker\cp_\theta\cap\ker\cp^*_\theta$. Then:
\begin{align*}
\LL_{B^{0,1}}\be&=\cp\iota_{B^{0,1}}\be+\iota_{B^{0,1}}\cp\be\\
&=\cp\iota_{B^{0,1}}\be+\be-\theta^{0,1}\wedge\iota_{B^{0,1}}\be
\end{align*}
which also reads:
\begin{equation*}
\LL_{B^{0,1}}\be-\be=\cp_\theta\iota_{B^{0,1}}\be.
\end{equation*}
Now the left hand side in the above equality is $\ov{\DA}_\theta$-harmonic, and thus, by \eqref{L2decomp}, $L^2$-orthogonal to the right hand side which belongs to the image of $\cp_\theta$, so both terms vanish. We infer:
\begin{equation*}
\LL_{B^{0,1}}|_{\ker\ov{\DA}_\theta}=\id, \ \  \LL_{B^{1,0}}|_{\ker\ov{\DA}_\theta}=-( \LL_{B^{0,1}}|_{\ker\ov{\DA}_\theta})^*=-\id
\end{equation*}
and therefore:
\begin{equation}\label{LLA}
\LL_A|_{\ker\ov{\DA}_\theta}=i(\LL_{B^{1,0}}-\LL_{B^{0,1}})|_{\ker\ov{\DA}_\theta}=-2i\id.
\end{equation}
If we denote by $\Phi_t$ the one-parameter group generated by $A$, then \eqref{LLA} implies:
\begin{equation}\label{PhiA}
\Phi_t^*\be=\e^{-2it}\be, \ \ \  \forall\be\in\ker\ov{\DA}_\theta .
\end{equation}

Let us now fix $\Omega\in[\Omega]$ with Lee form $\theta$. By hypothesis, there exists a smooth family of functions $f_t\in\ce(M)$ so that $\Phi_t^*\Omega=\e^{f_t}\Omega$, implying:
\begin{equation*}
\theta\wedge\Phi^*_t\Omega=\Phi^*_t\theta\wedge\Phi_t^*\Omega=d(\Phi^*_t\Omega)=(df_t+\theta)\wedge\Phi^*_t\Omega.
\end{equation*}

Let us suppose first that the complex dimension of $M$ is $n>1$. In this case, the above identity implies $\theta=df_t+\theta$, and since $M$ is compact, this implies then that $f_t=0$, i.e. $\Phi_t^*\Omega=\Omega$, $\forall t\in\RR$.

On the other hand, by \eqref{BCV} and \eqref{Hodgeiso},  there exist $f\in\ce(M,\RR)$ and $\al\in\Omega^{0,1}(M)$ with $\ov{\DA}_\theta\al=0$ so that
\begin{equation*}
\Omega=\del_\theta\al+\cp_\theta\ov{\al}+i\del_\theta\cp_\theta f.
\end{equation*}
Since for any $t\in\RR$, $\Phi_t$ is a biholomorphism of $(M,J)$ which preserves the form $\theta$, the action of $\Phi_t$ on smooth forms commutes  with the operators $\del_\theta$ and $\cp_\theta$. Thus we infer, via \eqref{PhiA}:
\begin{equation*}
\pi\cdot\Omega=\int_0^\pi\Phi_t^*\Omega dt=\del_\theta\int_0^\pi\Phi_t^*\al dt+\cp_\theta\int_0^\pi\Phi_t^*\ov{\al} dt+i\del_\theta\cp_\theta\int_0^\pi\Phi_t^*fdt=i\del_\theta\cp_\theta F
\end{equation*}
with $F=\int_0^\pi\Phi_t^*fdt$. Therefore $\Omega$ admits a potential. To see that $\Omega$ admits moreover an $A$-invariant potential, consider the closure $\TT_A$ of the one-parameter group $\Phi_t$ inside $\Aut(\Omega_0,J)$. Since the latter is a closed subgroup of the group of isometries of the fixed Vaisman metric, it is compact, and thus $\TT_A$ is a compact torus. As the group $\Phi_t$ preserves $\Omega$ and is dense in $\TT_A$, it follows that $\Omega$ is $\TT_A$-invariant. Also $\theta$ is $\TT_A$-invariant. Hence, letting $d\mu$ be the Haar measure on $\TT_A$ with $\int_{\TT_A}d\mu=1$, we have:
\begin{equation*}
\Omega=\int_{\TT_A}g^*\Omega d\mu(g)=\frac{i}{\pi}\del_\theta\cp_\theta\int_{\TT_A}g^*Fd\mu(g)
\end{equation*}
meaning that $\frac{1}{2\pi}\int_{\TT_A}g^*Fd\mu(g)$ is an $A$-invariant potential for $\Omega$.

In order to show uniqueness of $A$-invariant potentials, let us introduce the following sheaves over $M$: identify now $L_\theta$ with the sheaf of germs of real-valued smooth functions on $M$ which are $d_\theta$-closed. Let $\LL_\theta$ be the sheaf of germs of smooth complex-valued functions on $M$ which are $\cp_\theta$-closed. Finally, let $\mathcal{P}(L_\theta)$ be the sheaf of germs of smooth real-valued functions on $M$ which are $\del_\theta\cp_\theta$-closed. We have the short exact sequence of sheaves \cite[Section~5.2]{go14}:
\begin{equation*}
\xymatrix{ &0\ar[r] & L_\theta\ar[r] &\LL_\theta\ar[r] &\mathcal{P}(L_\theta)\ar[r] &0}
\end{equation*}
where the first map is simply the inclusion, while the latter consists in taking the imaginary part of a function. This then induces a long exact sequence in cohomology:
\begin{equation*}
\xymatrix{ &0\ar[r] & H^0(M,L_\theta)\ar[r] &H^0(M,\LL_\theta)\ar[r] &H^0(M,\mathcal{P}(L_\theta))\ar[r] & H^1(M,L_\theta)\ar[r] &\ldots
}
\end{equation*}

But $H^0(M,L_\theta)=0=H^1(M,L_\theta)$ because of the Vaisman hypothesis (\ref{Fexact}), hence we have an isomorphism:
\begin{equation}\label{kerdeldel}
\ker(i\del_\theta\cp_\theta|_{\ce(M,\RR)})=H^0(M,\mathcal{P}(L_\theta))\cong H^0(M,\LL_\theta).
\end{equation}

Let us now suppose that $\Omega$ admits two $A$-invariant potentials $f_1$ and $f_2$. It follows that $v:=f_2-f_1\in\ker(i\del_\theta\cp_\theta)$. By \eqref{kerdeldel}, there exists $\sigma\in\ce(M,\CC)$, $\sigma=u+iv$ with $u\in\ce(M,\RR)$, so that $\cp_\theta \sigma=0$. In particular, $\sigma\in\ker\ov{\DA}_\theta$, therefore by \eqref{LLA}, we have $A(\sigma)=-2i\sigma$, or also:
\begin{equation*}
A(u)=2v, \ \ A(v)=-2u.
\end{equation*}
From $A(v)=0$ we infer thus that $\sigma=0$, so $f_1=f_2$.

In the case $\dim_\CC M=1$, $(M,J)$ must be a torus by hypothesis,  so its canonical bundle $K_{M,J}$ is holomorphically trivial. Thus, by Serre duality we have:
\begin{equation*}
H^{0,1}(M,L_\theta)\cong H^1(M,\LL_\theta)\cong H^0(M,\LL_\theta^*)^*.
\end{equation*} 
Here we have identified $\LL_\theta$ with the complex line bundle $L_\theta\otimes\CC$ endowed with the holomorphic structure $\cp_L:=\cp-\theta^{0,1}\wedge\cdot$, and thus the first isomorphism in the above equation is simply the Dolbeault isomorphism.  Now, $\LL_\theta^*$ is not holomorphically trivial as $\theta^{0,1}$ is not $\cp$-exact. Hence, if there existed $\sigma\in H^0(M,\LL_\theta^*)$, it would vanish somewhere. On the other hand, the number of zeroes of $\sigma$ gives the first Chern class $c_1(\LL^*_\theta)\in H^2(M,\ZZ)\cong\ZZ$, which is zero as  $\LL^*_\theta$ is a flat line bundle. Therefore $\LL_\theta^*$ has no holomorphic section. By the same reason, neither does $\LL_\theta$ have any holomorphic section.  We conclude, by \eqref{BCV} and by \eqref{kerdeldel}, that any LCK metric $(\Omega,\theta)$ on $(M,J)$ has a unique potential. \end{proof}

Now we specialize to the toric context:

\begin{proof}[\textbf{Proof of \ref{toricPot}}] Let  $(M, J, [\Omega], [\theta]_{dR}, G)$ be a compact LCK manifold. By \ref{toricLCK}, there exists a $G$-invariant Vaisman metric $\Omega_0$ on $(M,J)$. Let $A$ denote its anti-Lee vector field $A$. By \cite[Lemma~4.7]{p16}, $A\in\Lg\subset\aut([\Omega])$, so \ref{APot} implies that any $G$-invariant form $\Omega\in[\Omega]$ admits a unique $G$-invariant potential $f$.

Let us now show that $f$ is positive. By \eqref{ideldelconf}, this is a conformally invariant property, so by the same reasoning as in the beginning of the above proof, we can suppose that $(\Omega,\theta)$ is chosen in the conformal class so that $\theta(B)=1$, where $B=-JA$. 

Let $\pi:\hat M\rightarrow M$ be the minimal cover of deck group $\Gamma$ and let $\pi^*\theta=d\phi$ on $\hat M$ so that
\begin{equation*}
\omega=\e^{-\phi}\pi^*\Omega=2i\del\cp \hat f=dJd\hat f, \ \ \ \hat f:=\e^{-\phi}\pi^*f.
\end{equation*} 
We will show that $\hat f$ is strictly positive on $\hat M$.

We recall, cf. Section \ref{SecClass}, that we have a vector field on $\hat M$ given by $d=d_\Lg+d_{J\Lg}\in\Lg\oplus J\Lg$  which generates the action of the deck group $\Gamma=\langle\gamma\rangle$ by $\gamma=\Phi_d^1$. Moreover,  we have $d_{J\Lg}=aB$, $a>0$ cf. \ref{theLee}. Denoting by $\nu=\Phi^a_B=\gamma\circ \Phi^{-1}_{d_\Lg}$ and using the fact that $\phi$ and $f$ are $G$-invariant, we obtain the following equivariance relations:
\begin{equation}\label{echiv}
\begin{aligned}
\nu^*\phi&=\gamma^*\phi=\phi+a\\
\nu^*\hat f&=\e^{-a-\phi}\gamma^*\pi^*f=\e^{-a}\hat f.
\end{aligned}
\end{equation}

Let us now fix $x\in\hat M$ and denote by $u:\RR\rightarrow\RR$ the function $u(t)=\hat f(\Phi_B^t(x))$. It is strictly convex, as we have, using that $\LL_A\hat f=0$:
\begin{align*}
\frac{d^2u}{dt^2}(t)&=\LL_B^2\hat f(y)\\
&=\LL_B(Jd\hat f(JB))(y)-\LL_{JB}(Jd\hat f(B))(y)\\
&=2\omega(B,JB)_{y}>0
\end{align*}
where $y=\Phi_B^t(x)$.

Let us first show that $\hat f$ has constant sign. If this is not the case, then we can  suppose that $x\in \hat M$ is so that $\hat f(x)=0$. But then we have, by  \eqref{echiv}:
\begin{equation*}
u(a)=\hat f(\nu(x))=\e^{-a}\hat f(x)=0=u(0).
\end{equation*}
It follows that $u(t)\leq 0$ $\forall t$, for otherwise $u$ would have a local maximum, which is impossible because $u$ is convex. But this then implies that $0$ is a maximal value of $u$, which is again impossible. Hence $\hat f$ has no zeroes.

Since $u$ is strictly convex, then $u$ either is  strictly monotone or has a global minimum $t_0$. But in the latter case, because of \eqref{echiv} and $a>0$, we have:
\begin{equation*}
u(t_0+a)=\e^{-a}u(t_0)<u(t_0)
\end{equation*}
which is impossible. This together with the fact that $u$ is of constant sign and convex then implies that $u>0$. In particular, $\hat f(x)=u(0)>0$, which holds for any choice of $x\in\hat M$, and this concludes the proof. \end{proof}

\end{document}